\documentclass[11pt]{amsart}
\usepackage{latexsym,amsmath,amssymb,amsfonts,amscd,graphics, mathtools}
\usepackage{blindtext}
\usepackage[a4paper, left=2.5cm,right=2.5cm, bottom=3cm, top = 3cm]{geometry}
\usepackage[mathscr]{eucal}
\usepackage{mathrsfs}
\usepackage{tikz-cd}
\usepackage[pagebackref]{hyperref}
\usepackage[capitalize]{cleveref}
\usepackage{thmtools}
\usepackage[colorinlistoftodos,prependcaption,textsize=tiny]{todonotes}
\hypersetup{
	colorlinks=true,
	linkcolor=blue,
	filecolor=blue,      
	urlcolor=blue,
	citecolor=teal!70!black}
\usepackage{libertine}

\makeatletter
  \def\th@definition{
  \thm@headfont{\bfseries} 
  \thm@notefont{\bfseries} 
}
\makeatother

\makeatletter
  \def\th@remark{
  \thm@headfont{\bfseries} 
  \thm@notefont{\bfseries} 
	}
\makeatother

\usepackage{setspace}

\newcommand{\sA}{\mathcal{A}}

\newcommand{\sC}{\mathcal{C}}

\newcommand{\sF}{\mathcal{F}}

\newcommand{\sJ}{\mathcal{J}}

\newcommand{\sM}{\mathcal{M}}

\newcommand{\sO}{\mathcal{O}}

\newcommand\sY{{\mathcal Y}}

\newcommand{\bbP}{\mathbb{P}}
\newcommand{\bbQ}{\mathbb{Q}}

\newcommand{\bbZ}{\mathbb{Z}}

\theoremstyle{plain}
\newtheorem{theorem}{Theorem}

\newtheorem{subprop}{Proposition}[theorem]

\newtheorem{proposition}[equation]{Proposition}
\newtheorem{corollary}[equation]{Corollary}

\theoremstyle{remark}
\newtheorem{remark}[equation]{Remark}

\theoremstyle{definition}
\newtheorem{definition}[equation]{Definition}
\newtheorem{subdef}{Definition}[definition]

\newcommand{\bP}{\mathbb{P}}

\newcommand{\bQ}{\mathbb{Q}}
\newcommand{\bZ}{\mathbb{Z}}

\newcommand{\calC}{\mathcal{C}}

\newcommand{\calF}{\mathcal{F}}

\newcommand{\calO}{\mathcal{O}}

\newcommand{\calY}{\mathcal{Y}}

\newcommand{\Jac}{\mathrm{Jac}}
\newcommand{\Sing}{\mathrm{Sing}}

\newcommand{\rank}{\mathrm{rk}}

\newcommand{\Pic}{\mathrm{Pic}}

\newcommand{\git}{/\kern-0.2em/}

\newcommand{\codim}{\mathrm{codim}}

\newcommand{\prim}{\mathrm{prim}}


\author{Yajnaseni Dutta}
\address{Mathematisch Instituut, Universiteit Leiden, Gorlaeus Gebouw, Einsteinweg 55, 2333 CC, Leiden, NL}
\email{y.dutta@math.leidenuniv.nl}

\author{Lisa Marquand}
\address{Courant Institute of Mathematical Sciences, New York University, 251 Mercer St, NY, 10012}
\email{lisa.marquand@nyu.edu}

\title[Relative Compactified Prym and Picard]{Relative Compactified Prym and Picard fibrations associated to very good cubic fourfolds} 

\begin{document}
	\begin{abstract}
    A very good cubic fourfold is a smooth cubic fourfold that does not contain a plane, a cubic scroll, or a hyperplane section with a corank 3 singularity. We prove that the normalization of the relative compactified Prym variety associated to the universal family of hyperplanes of a very good cubic fourfold is in fact smooth, thereby extending prior results of Laza, Sacc\`a and Voisin. Using a similar argument, we also prove the smoothness of the normalization of the relative compactified Picard of the associated relative Fano variety of relative lines. 
\end{abstract}
	\maketitle

\subsection{The introduction} Let $X\subset \bP^5$ be a smooth cubic fourfold, and let $p\colon \calY\rightarrow B\coloneqq(\bP^5)^\vee$ denote the universal family of hyperplane sections. 
Since the intermediate Jacobian of a smooth cubic threefold is a (principally polarized) Abelian variety, one may wonder whether akin to the case of relative compactified Jacobian of curves, such a compactified intermediate Jacobian family exists over $B$. This was first discussed in \cite{DonMar}.

Let $U_1 \coloneqq B\setminus \Sing(X^{\vee})$ be the open subset parametrising hyperplane sections of $X$ with at most one node, and $\pi_{U_1}\colon J_{U_1}\rightarrow U_1$ the fibration obtained by Donagi and Markman \cite{DonMar} by considering a version of the relative intermediate Jacobians. 
The fibration $\pi_{U_1}\colon J_{U_1}\rightarrow U_1$ is, moreover, a Lagrangian fibration: i.e. there exists a unique holomorphic symplectic $2$-form $\sigma_{U_1}$ on $J_{U_1}$, such that $\sigma_{U_1}|_{J_{U_1,b}}=0$ for all $b\in U_1$. 
It is a natural question to ask whether there exists a hyperk\"ahler compactification $\overline{J}$ of $J_{U_1}$ admitting a Lagrangian fibration with irreducible fibers extending $\pi_{U_1}$. 
Recall that hyperk\"ahler manifolds are simply connected even dimensional smooth complex projective varieties that admit a nondegenerate global holomorphic 2-form. It is expected that every hyperk\"ahler manifold admits a deformation with a Lagrangian fibration. The (numerical) property characterizing Lagrangian fibrations induces countable dense subsets in the moduli space of hyperk\"ahler manifolds \cite{mat3, kam-ver, DIKM}. Such a fibration mimics many properties enjoyed by elliptic fibrations of minimal smooth projective surfaces. Furthermore, any non-trivial fibration of a hyperk\"ahler manifold is Lagrangian, i.e. restricted to smooth points of the fibers, the 2-form vanishes.

In \cite{LSV}, the authors construct a smooth hyperk\"ahler compactification $\pi\colon\overline{J}\rightarrow B$ of $J_{U_1}$ such that $\pi$ has irreducible fibers, under the assumption that $X$ is \textbf{general}. The hyperk\"ahler manifold $\overline{J}$ is deformation equivalent to an $OG10$ manifold \cite[Corollary 6.3]{LSV} - such a manifold is said to be of $OG10$-type.
 Their construction relies on passing to Mumford's Prym variety interpretation of the intermediate Jacobian of a smooth hyperplane section of $X$ and compactifying the Jacobian as the moduli of rank one torsion free sheaves.
Without the assumption of the generality of $X$, it is not clear whether their construction would result in a smooth compactification (see \cite[Section 3]{LSV}). 
We prove
\begin{theorem}\label{thm:main}
    Let $X\subset \bP^5$ be a very good cubic fourfold as in Definition \ref{def:vg}.
    Then there exists a smooth projective $OG10$-type hyperk\"ahler compactification $\overline{J}$ of $J_{U_1}$ with a Lagrangian fibration $\pi\colon \overline{J}\rightarrow B$ extending $\pi_{U_1}$, such that 
 the fibers of $\pi$ are compactified Prym varieties of reduced, irreducible planar curves. In particular, the fibers are irreducible.
\end{theorem}

Let $Y$ be a hyperplane section of a cubic fourfold. Mumford's Prym construction for the intermediate Jacobian $J(Y)$ still holds for $Y$ mildly singular after a careful choice of line \cite[Section 2]{LSV} - such a line is called a very good line. We recall the definition:
\begin{definition}
    A line $\ell\subset Y$ in a cubic threefold $Y$ is \textbf{very good} if the associated conic bundle obtained by ${\rm Bl}_{\ell}Y\to \bbP^2$ has at worst nodal fibers and both the discriminant curve $C\subset \bbP^2$ and its \'etale double cover $\widetilde{C}$ are irreducible.
\end{definition}
Thus, for any smooth cubic fourfold $X$, a necessary condition for the relative Prym construction is the existence of such lines for every hyperplane section of $X$. This motivates our definition of a very good cubic fourfold:
\setcounter{definition}{2}
\begin{subdef}\label{def:vg}
    A \textbf{very good cubic fourfold} is a smooth cubic fourfold such that every hyperplane section $Y\coloneqq X\cap H$ has a very good line.
\end{subdef}

A hyperplane section of a general cubic fourfold, as shown in \cite[Section 2]{LSV}, will be mildly singular with allowable singularities in the sense of \cite[Definition 3.3]{LSV}, and have a very good line. 
The existence of very good lines in cubic threefolds with worse singularities was studied recently in \cite{marquand2024defect}. In particular, the authors show in \cite[Theorem 1.2, 1.3]{marquand2024defect} that if a cubic threefold $Y$ has no corank 3 singularities, then the existence of a very good line is equivalent to $Y$ not containing a plane or a rational normal cubic scroll.

    Recall that a cubic threefold with a corank 3 singularity at $[1:0:0:0:0]$ has equation $x_0x_1^2+g(x_1, x_2,x_3,x_4)=0$. The simplest example that can occur is an $\widetilde{E}_6$ singularity. In particular, the possible corank 3 singularities that can occur on cubic threefolds have Milnor number at least 8  \cite[Section 2]{V24}. A cubic threefold with a corank 3 singularity can never admit a very good line \cite[Remark 5.4]{marquand2024defect}, even if it does not contain a plane or a scroll. 
The locus of cubic fourfolds admitting such a hyperplane section has high codimension in the moduli of cubic fourfolds.

\begin{subdef}\label{def:vg2}
   A very good cubic fourfold can equivalently be defined as a smooth cubic fourfold not containing a plane, a cubic scroll, or a hyperplane section with a corank 3 singularity.
\end{subdef}

Under the assumption that $X$ is a very good cubic fourfold, in particular, that there is a very good line on every hyperplane, one can apply the relative Prym construction of \cite[Section 2]{LSV} to obtain a compactification $\overline{P}$ of $J_{U_1}$.
Unfortunately, the techniques of \cite{LSV} used to prove smoothness of $\overline{P}$  fail in this more general set up. 
Thus \Cref{thm:main} should be seen as an extension of that of \cite{LSV}, as it extends their construction to a likely wider subset of cubic fourfolds. 
Indeed, there exist smooth cubic fourfolds containing non-allowable singularities. The locus of cubic fourfolds admitting a $E_6$ singularity is now being investigated in an on-going joint work of the L.M.\ with Viktorova, where they show that a general such cubic fourfold has a unique algebraic surface class induced by the hyperplane section and hence is very good according to Definition \ref{def:vg2}.

In \cite{sac2021birational}, Sacc\`a proves the existence of a smooth hyperk\"ahler Lagrangian compactification for \textbf{any} smooth cubic $X$ via MMP techniques, at the expense of losing the explicit geometry of special fibers; 
for example, whether they are irreducible or not, is not clear. 
\cref{thm:main}, among other things, ensures such a compactification with integral fibers.

Conversely, if $X$ contains a plane or a cubic scroll, any smooth hyperk\"ahler compactification $\pi\colon\overline{J}\rightarrow B$ will have reducible fibers \cite[Theorem 1.2]{marquand2024defect},  \cite[Corollary 1.8]{Brosnan}, \cite[Remark 3.12]{sac2021birational}. 
Therefore, our result in fact shows the following.
\setcounter{corollary}{1}
\begin{corollary}
    Let $X$ be any smooth cubic fourfold, such that none of its hyperplane sections have a corank three singularity. Then a smooth hyperk\"ahler Lagrangian compactification with integral fibers of the relative intermediate Jacobian $J_{U_1}\to U_1$ exists if and only if $X$ is a very good cubic fourfold.
\end{corollary}
Note that since cubic threefolds with corank 3 singularities cannot have very good lines, our result, in particular the Prym construction, does not apply to smooth cubic fourfolds containing such hyperplane sections.

The next result is the crucial technical input that goes into proving \Cref{thm:main}, which is interesting in its own right.
\setcounter{theorem}{3}
\begin{subprop}\label{prop: Qfact}
    Let $X$ be a very good cubic fourfold. 
    Let $\overline{P}\to B$ be the relative compactified Prym associated to the universal hyperplane sections of $X$.
    The normalization $\overline{J}$ of $\overline{P}$ is a symplectic variety and has $\bQ$-factorial terminal singularities.
\end{subprop}
\Cref{thm:main} follows from \Cref{prop: Qfact} by applying a result of Greb--Lehn--Rollenske \cite[Proposition 6.5]{GLR}, that states that any terminal $\bQ$-factorial symplectic variety that is birational to a smooth hyperk\"ahler manifold is smooth itself. Indeed, $\overline{J}$ is birational to a smooth hyperk\"ahler manifold $M$, constructed using desingularising $\overline{P}$ and running a relative MMP, as was done in \cite{Sacca25} (later we call this the MMP model).
The last claims of \Cref{thm:main} follow from the construction of $\overline{P}.$ 

One can prove a variant of \Cref{thm:main}  by replacing the relative compactified Prym with a compactification of the scheme of relative degree 0 line bundles on the relative Fano variety of lines of the universal hyperplane section. This is given in \Cref{cor: pic} below. To prove this, we have a result analogous to \Cref{prop: Qfact}:
\setcounter{proposition}{2}
\begin{subprop}\label{prop: Qfactfano}
    Let $X$ be a very good cubic fourfold. 
    Let $f\colon \overline{\Pic}^0_{\sF/B}\to B$ be the  Picard model as described in \S \ref{sec:models}, where $\sF\to B$ is the relative Fano variety of lines associated to the universal hyperplane sections of $X$.
    The normalization of $\overline{\Pic}^0_{\sF/B}$ is a symplectic variety and has $\bQ$-factorial terminal singularities. 
\end{subprop}
 With the same argument as \cref{thm:main}, we obtain:

\setcounter{corollary}{3}
\begin{corollary}\label{cor: pic}
    The normalization of $\overline{\Pic}^0_{\sF/B}$ is a smooth hyperk\"ahler manifold birational to $\overline{J}$.
\end{corollary}
\setcounter{subsection}{1}

\begin{remark}
    In fact, when $X$ is a very good cubic fourfold, one can show that there is a unique smooth hyperk\"ahler compactification of $\pi_U\colon J_U\to U$. In particular, the models $\overline{J}, M$ and $\overline{\Pic}^0_{\calF/B}$ considered above are isomorphic as Lagrangian fibrations. Indeed, a recent result \cite[Theorem D]{DMS} shows that given such a compactification $\pi_M\colon M\to B$, there is an integral Hodge isometry: $${\rm H}^4_{\prim}(X,\bZ)(1)\oplus \langle \Theta_M, \pi_M^*\calO_B(1)\rangle\simeq {\rm H}^2(M,\bZ),$$
    where $\Theta_M$ is the relative theta divisor inducing a principal polarization on the smooth fibers of $\pi_M$. Similar isometries hold for $M$ replaced by any smooth hyperk\"ahler compactification of $\pi_{U}$.
   Then the isomorphisms of the models follow using the proof of \cite[Theorem 6.6]{MO}: Let $M_1$ and $M_2$ be two smooth hyperk\"ahler compactification of $\pi_U$, then there exists a Hodge isometry ${\rm H}^2(M_1,\bZ)\simeq {\rm H}^2(M_2,\bZ)$ 
    such that the class $\pi_{M_1}^*\sO_B(1)$ is sent to $\pi_{M_2}^*\sO_B(1)$, as well as the respective theta divisors $\Theta_{M_1}, \Theta_{M_2}$, being the closure of the same divisor on $J_U$, are sent to one another.
     If $M_1$ and $M_2$ were not isomorphic, as argued in \cite[Theorem 6.6]{MO}, a linear combination of $\Theta_{M_1}$ and $\pi_{M_1}^*\sO_B(1)$ will be ample on $M_1$, but not on $M_2$. They argue that the existence of such a class will imply that either $X$ is singular or not very good, contradicting our assumption on $X$.
\end{remark}
\subsection{The new input} 
Our main new input is \Cref{prop: sections}, and its analogue \cref{prop:hm} for relative $\Pic^0$ of the Fano variety. They use the results from \cite{DMS}. These results allow us to identify the sections of the relative compactified Prym $\overline{P}$ with the sections of a relative intermediate Jacobian sheaf $\sJ$. 
We use the sections to calculate the number of linearly independent Cartier divisors in $\overline{P}$ (see \cref{prop: pic rank}).
\subsection{The models} \label{sec:models}
We recall the construction of the currently available compactifications of $\pi_{U_1}:J_{U_1}\rightarrow B$, one constructed in \cite{LSV} and the other in \cite{sac2021birational}. The former has integral fibers but may not a-priori be smooth for any smooth cubic fourfold, whereas the latter is smooth for any cubic fourfold but may not a-priori have integral fibers. 
A-posteriori when $X$ is a very good cubic fourfold, the normalisation of the Prym model is isomorphic to the MMP model described below and, therefore, is smooth and has integral fibres.

Let $X\subset \bP^5$ be a very good cubic fourfold.
Let $\widetilde{p}\colon \calF\rightarrow B$ be the relative Fano variety of lines on $p:\calY\rightarrow B$, and let $\calF^0\subset \calF$ be the open set of very good lines.
By assumption, very good lines exist on $\calY_b$ for all $b\in B$, and thus the map $\calF^0\rightarrow B$ is smooth and surjective.
By considering the conic bundle associated to the blow up of $\calY_{\calF^0}$ along the universal very good line, and the double cover of the discriminant curve, we obtain families of irreducible curves $\widetilde{\sC}$ and $\sC$ over $\sF^0$ :
\[\begin{tikzcd}
    \widetilde{\calC}\arrow[rr, "g"]\arrow[dr]&&\calC\arrow[dl]\\
    & \calF^0&
\end{tikzcd},\]
such that over $\ell\in \sF^0$, the corresponding cover $\widetilde{\sC}_\ell\rightarrow \sC_\ell$ is an \'etale double cover of irreducible curves. 

We thus let $P_{\calF^0}={\rm Prym}(\widetilde{\calC},\calC)$ be the relative Prym variety and $\overline{P}_{\calF^0}=\overline{{\rm Prym}(\widetilde{\calC},\calC)}$ its natural compactification as considered in \cite{LSV}. Applying the descent argument \cite[\S 5]{LSV} to $\overline{P}_{\calF^0}\rightarrow \calF^0;$  one obtains
$$\pi\colon\overline{P}\rightarrow B$$
which compactifies $\pi_{U_1}$.
We call $\overline{P}$ \textbf{the Prym model}, which may a-priori be singular. 
By construction, $\overline{P}|_{U_1}$ provides a smooth fiberwise compactification of $J_{U_1}$. 
Note also that $\pi\colon\overline{P}\rightarrow B$ has irreducible fibers, and is equi-dimensional over the base by \cite[Prop 4.16]{LSV}.  

In \cite{sac2021birational}, Sacc\`a constructs a hyperk\"ahler compactification $M$ of $J_{U_1}$. We modify the construction for our purpose: we consider a desingularization of $\overline{P}$, and run the relative MMP outside the smooth locus given by $P\cup \overline{J}_{U_1}$,  where $P$ is the set of regular points along the fibers of $\pi$ and $\overline{J}_{U_1}$ is given by the compactified intermediate Jacobians of the at most 1-nodal hyperplane sections. Applying the arguments of \cite{Sacca25} (see also \cite[Section 1]{sac2021birational}), one proves that $M$ is smooth hyperk\"ahler manifold - we call $M$ an \textbf{MMP model.}

Finally, we will consider the \textbf{Picard model} associated to the family of Fano surfaces $\widetilde{p}\colon \sF\to B$, considered above. We let $\Pic^0_{\sF/B}$ be the scheme parametrising relative degree 0 line bundles. We denote by $\overline{\Pic}^0_{\sF/B}$ the closure of $\Pic^0_{\sF/B}$ in the projective scheme parameterising degree 0 torsion free sheaves of rank 1 . We consider $\overline{\Pic}^0_{\sF/B}$ with the reduced scheme structure. 
 There is an induced fibration $f\colon \overline{\Pic}^0_{\sF/B}\to B$; we call this the Picard model. 

 \begin{remark}\label{rmk:MMPfromPicard}
     Furthermore, we will see later in \cref{prop:hm} that $R^1\widetilde{p}_*\sO_{\sF}\simeq \Omega_B^1$. Since the fibers of $\widetilde{p}$ are integral \cite[Lemma 5.6]{marquand2024defect}, the sheaf $R^1\widetilde{p}_*\sO_{\sF}$ is the Lie algebra of the relative Picard scheme \cite[\S 8.4 Theorem 1(a)]{BLR90}. 
    Since $J_{U_1}$ is always smooth for any cubic fourfold, using the argument in \cite[Proposition 4.9]{Bottini} over $U_1$, one concludes that $\overline{\Pic}^0_{\sF/B}|_{_{U_1}}$ is isomorphic to $\overline{P}_{U_1}$. However, this argument does not extend to all of $B$, since we do not know whether $\overline{P}$ is smooth.
    Therefore an MMP model can also be constructed by considering a desingularization of $\overline{\Pic}^0_{\sF/B}$, and running the relative MMP outside the smooth locus given by $\Pic^0_{\sF/B}\cup \overline{\Pic}^0_{\sF/B}|_{_{U_1}}$.
    This model is birational to the MMP model constructed from the Prym model.
 \end{remark}

\subsection{The Prym model}
In order to prove \Cref{prop: Qfact}, we will first analyze the sections of the fibration $\pi\colon\overline{P}\rightarrow B$. When $\overline{P}$ is smooth, this was observed in \cite[Theorem 5.1]{sac2021birational}.

\begin{proposition}\label{prop: sections}
    The map $\pi\colon \overline{P}\rightarrow B$ has sections and the group of sections is isomorphic to the primitive cohomology ${\rm H}^{2,2}_{\prim}(X, \bZ)$ of the very good cubic $X$.
\end{proposition}
\begin{proof}
    Let $P_{\sF_0}$ denote the set of regular points along the fibers of $\overline{P}_{\sF_0}$. Since the families of curves $\widetilde{\sC}\to \sF_0$ and $\sC\to \sF_0$ have irreducible fibers, $P_{\sF_0}$ is a group scheme and by \cite[\S 9.1]{ACLS-standard} descends to a group scheme $P$ over $B$ given by the set of regular points along the fibers of of $\overline{P}$. Recall that $\overline{P}_{\sF_0}\subset \overline{\Jac}(\tilde{\calC})$ is constructed as the connected component of the fixed locus of an involution of $\overline{\Jac}(\tilde{\calC})$ that contains the zero section \cite[Definition 4.8]{LSV}. It follows that $\pi\colon\overline{P}\rightarrow B$ has a zero section passing through $P$.

    In order to compute the exact number of sections, 
   we use the action of the relative Albanese $A_{M/B}$ \cite{Markushevich, AriFed, AbashevaRogov, kim24, Sacca25, AbashevaSTII} associated to the MMP model $\pi_M\colon M\to B$. 
    This can be defined as the fiberwise connected neutral component of the relative automorphism space of $\pi_M$.
    By construction, we have $P\subset M$, and hence the Lagrangian fibration $\pi_M\colon M\to B$ admits a reduced and irreducible component containing $P$ on each fiber. 
    Thus, all fibers of $\pi_M$ meet the zero section.
    Using \cite[Theorem 4.2, Proposition 4.27]{kim24} or \cite[Proposition 5.10]{Sacca25}, we thus conclude that $A_{M/B}$ acts freely on the locus $M^{\rm sm}$ of regular points of $\pi_M$. 
    Since $A_{M/B}$ has connected fibers, the locus $P\subset M^{\rm sm}$ forms a trivial $A_{M/B}$-torsor, in particular, $A_{M/B}=P\subset M$.
    Let $\sA_{M/B}$ denote the relative Albanese sheaf given by the sheaf of analytic sections of $A_{M/B}$.
    We conclude that ${\rm H}^0(B, \sA_{M/B})$ computes the group of sections of $\pi\colon \overline{P}\to B$. 
    Finally, using \cite[Corollary 6.8, Corollary 3.22]{DMS} we obtain that ${\rm H}^0(B,\sA_{M/B})\simeq {\rm H}^{2,2}_{\prim}(X,\bbZ).$
\end{proof}

Next, we will use the sections of $\pi$ to construct enough $\bQ$-Cartier divisors on $\overline{P}.$
\begin{proposition}\label{prop: pic rank}
      
  The number of linearly independent Cartier divisors on $\overline{P}$ is at least $$\rank {\rm H}^{2,2}_{\prim}(X,\bZ)+2,$$
  where $\rank{\rm H}^{2,2}_{\prim}(X,\bbZ)$ denotes the rank of the algebraic primitive cohomology.
\end{proposition}
\begin{proof}
    The variety $\overline{P}$ is projective, so let $L$ be some ample line bundle on $\overline{P}.$ Recall that $\overline{P}\subset \overline{\Jac}(\tilde{\calC})$. A section $s\colon B\rightarrow \overline{P}$ thus acts on $\overline{P}$ by translations fiberwise, i.e, it acts by tensoring. We denote this by $t_s\colon\overline{P}\rightarrow \overline{P}$.

    Consider the set $\{t_s^*L\}_{s\in H^{2,2}_{\prim}(X,\bbZ)}$ of line bundles. 
    Let $s=\sum_ia_is_i$ such that $\{s_i\}$ is a basis of the free $\bbZ$-module $H^{2,2}_{\prim}(X,\bbZ)$. 
    We claim that $t_s^*L \simeq L$ if and only if $a_i = 0$. To see this, 
    let $A$ denote the generic fiber of $\pi\colon \overline{P}\rightarrow B$. Then $A$ is an abelian variety with $L|_A$ ample, and $t^*_sL|_A$ is the translation of $L|_A$ via the section. 
    The claim follows from noting that 
    the map $\phi\colon A\rightarrow \hat{A},$ defined by $s\mapsto t_s^*L|_A\otimes L|_A^{-1}$ has finite kernel and each $s\neq 0$ has infinite order. 
    Therefore, the set $\{t_{s_i}^*L\}$ of Cartier divisors on $\overline{P}$ is linearly independent.

    In addition to these, there is the ample line bundle $L$, as well as the line bundle $\pi^*\calO_B(1)$. Their linear independence from the set of line bundles above can be checked again by restricting to $A$ and using the theorem of the square \cite[\S 6, Corollary 4]{Mum}. 
    Indeed, if $L|_A\simeq t_{s}^*L|_A\otimes t_{s'}^*L|_A$, for $s,s'\in H^{2,2}_{\prim}(X,\bbZ)$, then by the theorem of square again, $t_{s+s'}^*L|_A \simeq \sO_A$, which is absurd.    
    Similarly, since $\pi^*\sO_B(1)|_A \simeq \sO_A$, if for $s,s'\in H^{2,2}_{\prim}(X,\bbZ)$, we had $\sO_A\simeq t_s^*L|_A\otimes t_{s'}^*L|_A$, then by the theorem of square we have $t_{s+s'}^*L|_A \otimes L|_A\simeq \sO_A$. Since $s+s'$ have infinite order this is not possible.
\end{proof}

Next, we prove that the normalization $\overline{J}$ of $\overline{P}$ is a symplectic variety and use \cref{prop: pic rank} above to conclude that $\overline{J}$ is also $\bbQ$-factorial. 
\begin{proof}[Proof of \Cref{prop: Qfact}]
    Let $\overline{J}\rightarrow \overline{P}$ be the normalization. 
    We claim that $\overline{J}$ is a symplectic variety in the sense of Namikawa \cite{Nam} and has terminal singularities. 
    To see the claim, first note that $\codim_{\overline{P}}(\Sing(\overline{P})) \geq 3.$
    This can be seen as follows:  the Prym model $\overline{P}$ contains $\overline{J}_{U_1}$, the relative compactified Jacobian of the at most 1-nodal hyperplane sections of $X$, which is contained in the regular locus. This has a complement of codimension 2. 
    Therefore, the singular points of $\overline{P}$ must lie above $B\setminus U_1$.
    However the regular locus of the Prym model $\overline{P}$ also contains a smooth open subset $P$, consisting of the regular points of the fibers of $\pi$. 
    Since by construction, $\pi\colon\overline{P}\rightarrow B$ is equi-dimensional (see \cite[Proposition 4.10]{LSV}), and complement of $P$ in $\overline{P}$ on each fiber is at most 1, and hence the codimension bound follows.
    Secondly, the non-degenerate holomorphic 2-form on $\overline{J}_{U_1}$ extends to the regular locus of $\overline{P}$ as a non-degenerate $2$-form.
    Indeed, by construction the regular locus of $\overline{P}$ is contained in the MMP model $M$. 
    One obtains the desired 2-form by restricting the non-degenerate 2-form from $M$ to the regular locus of $\overline{P}$. Finally, since $\overline{J}$ is normal, 
    the singularities are symplectic. It follows from  \cite[Theorem 1, Corollary 1]{Nam} that the singularities are terminal.

    It remains to argue that $\overline{J}$ is $\bbQ$-Cartier. Let $D$ be a Weil divisor on $\overline{J}$.  
    Since $\overline{J}$ is birational to $M$, by \cite[Proposition 4.3]{Sacca25} $\overline{J}$ admits a symplectic resolution $\mu\colon M'\rightarrow \overline{J}$ with $M'$ birational to $M$. Then $\mu^*D$ is a Cartier divisor on $M'$. Since a birational map between smooth hyperk\"ahler manifolds is an isomorphism in codimension 1, the Picard rank $\rho(M')$ equals $\rho(M)$. By \cite[Lemma 3.2]{sac2021birational}, we know that $\rho(M)=\rank({\rm H}^{2,2}_{\prim}(X, \bZ))+2$. 
    By \Cref{prop: pic rank}, we can find $(\rank H^{2,2}_{\prim}(X,\bbZ)+2)$-many linearly independent Cartier divisors in $\overline{J}$. 
    Therefore, $\Pic(M')$ is generated by $\mu^*L$, $\mu^*\pi^*\sO_B(1)$ and $\mu^*(t_s^*L)$, for $s\in H^{2,2}_{\prim}(B,\bbZ)$. 
    It follows that $\mu^*D$ is a linear combination of these Cartier divisors. This relation also holds over 
    the smooth locus $\overline{J}^{\rm sm}$.
    Since $\codim_{\overline{J}}(\mathrm{Sing}(\overline{J}))\geq 4$, we have an isomorphism of the divisor class groups ${\rm Cl}(\overline{J}^{\rm sm}) \simeq {\rm Cl}(\overline{J})$. Hence the linear relation of $D$ holds over $\overline{J}$, which makes $D$ already Cartier on $\overline{J}$. 
\end{proof}

\subsection{The Picard model}
We will need the following proposition identifying the sections of the fibration $f\colon \Pic^0_{\sF/B}\rightarrow B$. It may be proved exactly analogously to \cref{prop: sections}, however we give a different Hodge theoretic proof since it avoids going via any hyperk\"ahler models. 
For a quick introduction to the Hodge modules we need and their properties adapted to our set-up consult \cite[\S 5.1]{DMS}.

\begin{proposition}\label{prop:hm}
     The coherent sheaf of Lie algebras of the group scheme $\Pic^0_{\sF/B}$ is isomorphic to $\Omega_B^1$. Furthermore, the map $f\colon \overline{\Pic}^0_{\sF/B}\to B$ has sections and the group of sections is isomorphic to ${\rm H}^{2,2}_{\prim}(X,\bbZ)$. 
\end{proposition}
\begin{proof}
    Let $U\subset B$ be the locus over which the fibers of $p\colon\sY\to B$ and $\widetilde{p}\colon \sF\to B$ are smooth.
    Consider the following isomorphic polarized variations of Hodge structure (pvHs) underlying the lattices $R^1\widetilde{p}_*\bbZ_{\sF_U}$ and $R^3p_*\bbZ_{\sY_U}(1)$. Let $\sM_{\sF}$ and $\sM_{\sY}$, respectively, denote the minimal extension (see e.g.\ \cite[\S 5.1 (1)]{DMS}) of this pvHs as Hodge modules. 
    Since the Hodge modules are isomorphic, we have
    \[\sJ(\sM_{\sF})\simeq \sJ(\sM_{\sY}),\]
    where $\sJ(\sM)$ denotes the Jacobian sheaf associated to the Hodge module $\sM$ (see [Definition 5.5, \textit{loc.\ cit.}]). 
    Furthermore, it follows, from \cite[Theorem 5.8]{DMS} that $\sJ(\sM_{\sF})\simeq \sJ_{\sF/B}$, where $\sJ_{\sF/B}$ is the sheaf of analytic sections of $\Pic^0_{\sF/B}$.
    In particular, we obtain an isomorphism of the corresponding sheaf of Lie algebras $R^1\widetilde{p}_*\sO_{\sF}\simeq R^2p_*\Omega_{\sY}^1\simeq \Omega_B^1$ [Theorem 5.10, Proposition 3.7 \textit{loc.\ cit.}]. This proves the first part.

    To compute the exact number of sections follows, as before, first observe that $\sJ_{\sF/B}\simeq~\sJ(\sM_{\sY})$, which follows from the above discussion. 
    Since very good cubic fourfolds are, in particular, defect general in the sense of
\cite[Definition 3.8]{DMS}, we apply [Corollary 3.22, \textit{loc.\ cit}] to conclude that \[{\rm H}^0(B, \sJ_{\sF/B})\simeq {\rm H}^{2,2}_{\prim}(X,\bbZ).\qedhere\] 
\end{proof}

The proof of \cref{prop: Qfactfano} is verbatim to \cref{prop: Qfact}. However, there are a few subtleties arising from considering compactified $\Pic^0$ of a family of surfaces (the Fano surfaces). Many of these were taken care of in \cite{Bottini}, which we point out in the argument below.
\begin{proof}[Proof of \cref{prop: Qfactfano}]
    
     To see that the normalization of $\overline{\Pic}^0_{\sF/B}$ is a terminal symplectic variety, here one needs to use 
     \cite[Proposition 4.8]{Bottini}, to make sure that the smooth group scheme $\Pic^0_{\sF/B}$ (see \cref{rmk:MMPfromPicard}) has complement of codimension 2 in $\overline{\Pic}^0_{\sF/B}$.
     The statement in \textit{loc.\ cit. } holds for any very good cubic fourfold.
    
    For $\bbQ$-factoriality, we need to use \cref{prop:hm}. Let $\sJ_{\sF/B}$ be the sheaf of analytic sections of $\Pic^0_{\sF/B}$.
    As in the proof of \cref{prop: pic rank}, one can argue again that $\overline{\Pic}^0_{\sF/B}$ has at least $\rank ({\rm H}^2(B, \sJ_{\sF/B}))$-many independent Cartier divisors in addition to $f^*\sO(1)$ and an ample line bundle $L$ coming from the projectivity of $\overline{\Pic}^0_{\sF/B}$. 
    Therefore, by \cref{prop:hm} we have $\rank ({\rm H}^{2,2}_{\prim}(X,\bZ))$-many sections of $f$. 
    Using the same argument as \cref{prop: pic rank}, we can produce $\rank ({\rm H}^{2,2}_{\prim}(X,\bZ))+2$-many Cartier divisors on $\overline{\Pic}^0_{\sF/B}$. This allows us to conclude that $\overline{\Pic}^0_{\sF/B}$ is $\bQ$-Cartier by applying the same argument as in the end of \Cref{prop: Qfact}.
\end{proof}

\begin{remark}
    Bottini in \cite[Remark 4.11.]{Bottini} showed that $\overline{\Pic}^0_{\sF/B}$ is a smooth hyperk\"ahler manifold when $X$ is a general cubic fourfold. In particular, at least for the general cubic fourfold, the normalization seems unnecessary. 
\end{remark}

\subsection{The Acknowledgments} We thank Yoonjoo Kim, J\'anos Koll\'ar and Giulia Sacc\`a for helpful discussions, email correspondences and feedback. In addition, we thank Alessio Bottini, Dominique Mattei,  Evgeny Shinder, and Sasha Viktorova for useful feedback. Y.D.\ would like to thank Dominique Mattei and Evgeny Shinder for the journey through several exciting mathematics during preparation of \cite{DMS} which are closely related to this note. We gratefully acknowledge the referee's valuable inputs in improving the exposition.

L.M. was supported in part by an AMS-Simons travel grant.

\bibliographystyle{halpha}
\bibliography{bibliography}

\begin{thebibliography}{DIKM24}
\expandafter\ifx\csname url\endcsname\relax
  \def\url#1{\texttt{#1}}\fi
\expandafter\ifx\csname doi\endcsname\relax
  \def\doi#1{\burlalt{doi:#1}{http://dx.doi.org/#1}}\fi
\expandafter\ifx\csname urlprefix\endcsname\relax\def\urlprefix{URL }\fi
\expandafter\ifx\csname href\endcsname\relax
  \def\href#1#2{#2}\fi
\expandafter\ifx\csname burlalt\endcsname\relax
  \def\burlalt#1#2{\href{#2}{#1}}\fi

\bibitem[Aba24]{AbashevaSTII}
Anna Abasheva.
\newblock Shafarevich--{Tate} groups of holomorphic {Lagrangian} fibrations
  {II}, 2024, \burlalt{2407.09178v3}{http://arxiv.org/abs/2407.09178v3}.

\bibitem[ACLS23]{ACLS-standard}
Giuseppe Ancona, Mattia Cavicchi, Robert Laterveer, and Giulia Saccà.
\newblock Relative and absolute {L}efschetz standard conjectures for some
  {L}agrangian fibrations, 2023,
  \burlalt{2304.00978}{http://arxiv.org/abs/2304.00978}.

\bibitem[AF16]{AriFed}
Dima Arinkin and Roman Fedorov.
\newblock Partial {Fourier}-{Mukai} transform for integrable systems with
  applications to {Hitchin} fibration.
\newblock {\em Duke Math. J.} {\bfseries 165}(15), pp.\ 2991--3042, 2016.
\newblock \doi{10.1215/00127094-3645223}.

\bibitem[AR21]{AbashevaRogov}
Anna Abasheva and Vasily Rogov.
\newblock Shafarevich--{Tate} groups of holomorphic {Lagrangian} fibrations,
  2021, \burlalt{2112.10921}{http://arxiv.org/abs/2112.10921}.

\bibitem[BLR90]{BLR90}
Siegfried Bosch, Werner L{\"u}tkebohmert, and Michel Raynaud.
\newblock {\em N{\'e}ron models}, volume~21 of {\em Ergeb. Math. Grenzgeb., 3.
  Folge}.
\newblock Berlin etc.: Springer-Verlag, 1990.
\newblock \doi{10.1007/978-3-642-51438-8}.

\bibitem[Bot25]{Bottini}
Alessio Bottini.
\newblock O'{G}rady's tenfolds from stable bundles on hyper-{K}\"ahler
  fourfolds, 2025, \burlalt{2411.18528}{http://arxiv.org/abs/2411.18528}.

\bibitem[Bro18]{Brosnan}
Patrick Brosnan.
\newblock Perverse obstructions to flat regular compactifications.
\newblock {\em Math. Z.} {\bfseries 290}(1-2), pp.\ 103--110, 2018.
\newblock \doi{10.1007/s00209-017-2010-0}.

\bibitem[DIKM24]{DIKM}
Yajnaseni Dutta, Elham Izadi, Ljudmila Kamenova, and Lisa Marquand.
\newblock Some density results for hyperk\"ahler manifolds, 2024,
  \burlalt{2403.04868}{http://arxiv.org/abs/2403.04868}.
\newblock \urlprefix\url{https://arxiv.org/abs/2403.04868}.

\bibitem[DM96]{DonMar}
Ron Donagi and Eyal Markman.
\newblock {\em Spectral covers, algebraically completely integrable,
  hamiltonian systems, and moduli of bundles}, pages 1--119.
\newblock Springer Berlin Heidelberg, Berlin, Heidelberg, 1996.
\newblock \doi{10.1007/BFb0094792}.

\bibitem[DMS24]{DMS}
Yajnaseni Dutta, Dominique Mattei, and Evgeny Shinder.
\newblock Twists of intermediate {J}acobian fibrations, 2024,
  \burlalt{2411.01953v2}{http://arxiv.org/abs/2411.01953v2}.

\bibitem[GLR13]{GLR}
Daniel Greb, Christian Lehn, and S\"{o}nke Rollenske.
\newblock Lagrangian fibrations on hyperk\"{a}hler manifolds---on a question of
  {B}eauville.
\newblock {\em Ann. Sci. \'{E}c. Norm. Sup\'{e}r. (4)} {\bfseries 46}(3), pp.\
  375--403, 2013.
\newblock \doi{10.24033/asens.2191}.

\bibitem[Kim24]{kim24}
Yoon-Joo Kim.
\newblock The {N}\'eron model of a higher-dimensional {L}agrangian fibration,
  2024, \burlalt{2410.21193}{http://arxiv.org/abs/2410.21193}.

\bibitem[KV14]{kam-ver}
Ljudmila Kamenova and Misha Verbitsky.
\newblock Families of {Lagrangian} fibrations on hyperk{\"a}hler manifolds.
\newblock {\em Adv. Math.} 260, pp.\ 401--413, 2014.
\newblock \doi{10.1016/j.aim.2013.10.033}.

\bibitem[LSV17]{LSV}
Radu Laza, Giulia Sacc\`a, and Claire Voisin.
\newblock A hyper-{K}\"{a}hler compactification of the intermediate {J}acobian
  fibration associated with a cubic 4-fold.
\newblock {\em Acta Math.} {\bfseries 218}(1), pp.\ 55--135, 2017.
\newblock \doi{10.4310/ACTA.2017.v218.n1.a2}.

\bibitem[Mar96]{Markushevich}
Dimitri Markushevich.
\newblock Lagrangian families of {J}acobians of genus {$2$} curves.
\newblock {\em J. Math. Sci.} {\bfseries 82}(1), pp.\ 3268--3284, 1996.
\newblock \doi{10.1007/BF02362472}.
\newblock Algebraic geometry, 5.

\bibitem[Mat17]{mat3}
Daisuke Matsushita.
\newblock On isotropic divisors on irreducible symplectic manifolds.
\newblock In {\em Higher dimensional algebraic geometry. In honour of Professor
  Yujiro Kawamata's sixtieth birthday. Proceedings of the conference, Tokyo,
  Japan, January 7--11, 2013}, pages 291--312. Tokyo: Mathematical Society of
  Japan (MSJ), 2017.

\bibitem[MO22]{MO}
Giovanni Mongardi and Claudio Onorati.
\newblock Birational geometry of irreducible holomorphic symplectic tenfolds of
  {O}'{G}rady type.
\newblock {\em Math. Z.} {\bfseries 300}(4), pp.\ 3497--3526, 2022.
\newblock \doi{10.1007/s00209-021-02966-6}.

\bibitem[Mum08]{Mum}
David Mumford.
\newblock {\em Abelian varieties}, volume~5 of {\em Tata Institute of
  Fundamental Research Studies in Mathematics}.
\newblock Tata Institute of Fundamental Research, Bombay; by Hindustan Book
  Agency, New Delhi, 2008.
\newblock With appendices by C. P. Ramanujam and Yuri Manin, Corrected reprint
  of the second (1974) edition.

\bibitem[MV23]{marquand2024defect}
Lisa Marquand and Sasha Viktorova.
\newblock The defect of a cubic threefold, 2023,
  \burlalt{2312.05118}{http://arxiv.org/abs/2312.05118}.

\bibitem[Nam01]{Nam}
Yoshinori Namikawa.
\newblock A note on symplectic singularities.
\newblock {\em arXiv} , 2001,
  \burlalt{math.AG/0101028}{http://arxiv.org/abs/math.AG/0101028}.

\bibitem[Sac23]{sac2021birational}
Giulia Sacc{\`a}.
\newblock Birational geometry of the intermediate {J}acobian fibration of a
  cubic fourfold.
\newblock {\em Geom. Topol.} {\bfseries 27}(4), pp.\ 1479--1538, 2023.
\newblock \doi{10.2140/gt.2023.27.1479}.

\bibitem[Sac24]{Sacca25}
Giulia Saccà.
\newblock Compactifying {L}agrangian fibrations, 2024,
  \burlalt{2411.06505}{http://arxiv.org/abs/2411.06505}.

\bibitem[Vik24]{V24}
Sasha Viktorova.
\newblock On the classification of singular cubic threefolds, 2024,
  \burlalt{2304.10452}{http://arxiv.org/abs/2304.10452}.
\newblock To appear: Trans. AMS.

\end{thebibliography}
\end{document}